\documentclass{amsart}

\newtheorem{theorem}{Theorem}[section]
\newtheorem{lemma}[theorem]{Lemma}
\newtheorem{proposition}[theorem]{Proposition}
\theoremstyle{definition}
\newtheorem{definition}[theorem]{Definition}
\newtheorem{example}[theorem]{Example}

\theoremstyle{remark}
\newtheorem{remark}[theorem]{Remark}

 \numberwithin{equation}{section}
\usepackage{graphicx}
\usepackage{placeins}

\begin{document}

\title{AN INDEPENDENCE SYSTEM AS KNOT INVARIANT}

\author{Usman Ali}
\address{Center for Advanced Studies in Pure and Applied
Mathematics, Bahauddin Zakariya University, Multan, Pakistan.}

\email{uali@bzu.edu.pk}

\author{Iffat Fida Hussain}
\address{Center for Advanced Studies in Pure and Applied
Mathematics, Bahauddin Zakariya University, Multan, Pakistan.}
\email{tariq75786@gmail.com}


\date{....., .... and, in revised form, ....., .....}



\begin{abstract}
In this article, we define an independence system for a classical
knot diagram and prove that the independence system is a knot
invariant for alternating knots. We also discuss the exchange property for minimal
unknotting sets. Finally, we show that there are knot diagrams where the independence system is a
matroid and there are knot diagrams where it is not.
\end{abstract}

\maketitle

%
.

\section{Introduction}

When we draw knot diagrams, we can define what is called a U-independence system for it.  The U-independence system can be used to help find new invariants
for alternating knots. A more motivating factor for the U-independence system is its usefulness in examining the relationships between knots and combinatorial
 objects like matroids. When a U-independence system for a knot diagram is a matroid, we can say that every maximal U-independent set has the same cardinality.
  As a result, every minimal unknotting set has the same minimal cardinality. In other words, we need only find a minimal unknotting set in order to determine
  the unknotting number of a knot diagram. This makes the algorithmic methods of finding the unknotting number of a knot much simpler and quicker. This paper
  delves further into the definition of the U-independent set and U-independence system in Section 2. In Section 3, we provide more basic information and
  examples of independence systems and matroids as well as discuss the exchange property for minimal unknotting sets. Next, in Section 4,
  we discuss the properties of a U-independence system and provide a proof of an existence of
isomorphisms between two U-independence systems of reduced
alternating diagrams of a knot. Section 4 will also highlight how
invariants of the U-independence system can be used as invariants of
knots. Finally, we conclude this paper with a proof of the various
relationships between a U-independence system of a knot diagram and
matroids in different families. We further extend our research to
other open areas of research, including bridge numbers and algebraic
unknotting numbers.
\section{Definitions and Examples of Basic Notations}

To first define what a U-independent set is, we begin with an understanding of unknotting numbers and unknotting sets.

The \textit{unknotting number} $u(D)$ of a knot diagram $D$ is the minimum number of switches required to untangle that particular knot diagram. In contrast, the
\textit{unknotting number} $u(K)$ of a knot $K$ is the minimum number of crossings required to switch to the unknot that ranges over all possible diagrams of
knot K.

An \textit{unknotting set} for a knot diagram is the set of all
switches that transforms the diagram into the unknot. We define the
\textit{minimal unknotting set} as the set of switches that have no
proper unknotting subsets. In other words, the minimal unknotting
set contains the minimum number of switches needed to transform the
particular diagram into an unknot.

\begin{definition}\label{exch}

Minimal unknotting sets for a knot diagram have the exchange
property in whenever {S} and {R} are two minimal unknotting sets and
$r\in R$ then there exists $s\in S$ so that $S -  \{s\} \cup \{r\}$
is a minimal unknotting set.
\end{definition}
In simpler terms, the exchange property is said to be true if we can
remove any one element from a minimal unknotting set S and replace
it with another element from some other minimal unknotting set R so
that the resulting set is also a minimal unknotting set. The
exchange property raises some certain advantages: if the exchange
property holds for two sets, then every minimal unknotting set for
that diagram has the same size. This makes algorithmic methods to
find the minimum size of an unknotting set easier. For example, the
exchange property for the diagram of the figure eight knot holds
because all the minimal unknotting sets are of cardinality one (Fig.
1(a)). However, to show that the exchange property does not hold, we
can show that there exists two minimal unknotting sets with
different cardinalities. For example, the three twist knot (Fig.
1(b)) has the following minimal unknotting sets: $$ \{v_{4}\},\
\{v_{5}\},\{v_{1},v_{2}\},\{v_{1},v_{3}\} \textrm{ and }
\{v_{2},v_{3}\}.$$ Since the minimal unknotting sets do not have the
same cardinality, the exchange property does not hold. There are
cases, however, in which the exchange property still does not hold
for minimal unknotting sets with the same cardinality (see
subsection 3.3 for further explanation).

A property defined in a finite set, which is also a property of
its subsets, is called a \textit{hereditary property}, see \cite{12}.
An \textit{independence family} $I $ on a finite ground set $E$ is a
non-empty collection of sets $X\subset E$, satisfying the hereditary
property. An \textit{independence system} $(E,I)$ for the set $E$ consists of an independence family $I$ with subsets of $E$.
The maximal independent sets are called \textit{bases} of $(E,I)$.
An independence system is called a \textit{matroid} if all of its
bases have the exchange property (see \cite[Definition 2.1]{12}).\\

With these terms in mind, we can formulate our key definition.

\begin{definition}\label{key}
A U-independent set is a set $W$ of crossings in a given knot
diagram such that $W\setminus S$ is not an unknotting set for every
non-empty $S\subseteq W$. In other words, a U-independent set is the
set of crossings that does not contain an unknotting set.
\end{definition}

The definition of a U-independent set leads to the U-independence system $%
(E,I)$ for a knot diagram $D$, where $E$ is the set of all crossings
of $D$, and $I$ is the independence family consisting of the
U-independent sets for $D$. In other words, the independence system
$(E,I)$ is the set of all subsets of $E$ that do not contain a
proper unknotting set.

We say that a U-independent set is \textit{maximal} if it is not
contained in any other U-independent set. By the definition of
U-independence, every minimal unknotting set is a maximal
U-independent set. While this statement is true, its converse (every
maximal U-independent set is a minimal unknotting set) may not
always be true.

\begin{definition}\label{reduced}
A \textit{reduced knot diagram} is a knot diagram where no crossing can be removed just by twisting it.
\end{definition}
\begin{definition}\label{minimald}
A \textit{minimal knot diagram} is a knot diagram which needs the
minimum crossings to draw the knot.
\end{definition}
A minimal knot diagram in a Rolfsen table is denoted by $m_t$ where
$m$ is the number of crossings in the diagram and $t$ is the number
of different knot diagrams with $m$ crossings, see \cite{10}. An
\textit{alternating knot} is a knot which has a knot diagram in
which crossings alternate under and over each other. In case of
alternating knots, the notions of minimal and reduced diagrams
coincide. Consequently, all the reduced alternating diagrams of a
knot have the same number of crossings, see \cite{K6}. The
U-independence system for a reduced alternating diagram of a knot is
declared as a knot invariant by the following theorem (see Section
4.2 for its proof):

\begin{theorem}\label{iso}
Let ($E_{1},I_{1})$ and ($E_{2},I_{2})$ be the  U-independence
systems of reduced alternating diagrams $D_1$ and $D_2$ of a knot,
$K$ respectively. There exists an isomorphism $\varphi$ between
($E_{1},I_{1})$ and ($E_{2},I_{2})$.
\end{theorem}

Knots given by $c_{1}, c_{2},\ldots, c_{j}$ in Conway notation are
denoted by $(c_{1}, c_{2},\ldots c_{j})$, see \cite{1,4}. The
following proposition describes whether the U-independence systems
of knot diagrams $(2n+1, 1, 2n)$, $(2n+1)$, and $(2n,2)$ are
matroids or not.

\begin{proposition}\label{fig15}
The U-independence system of each knot diagram:\\
 a) is not a matroid in the family $(2n+1,1,2n)$ for $n\geq \ 2$ (Fig.
$7$);\\
b) is a matroid in the family $(2n+1)$ for
$n\geq 1$ (Fig. $9$); and\\
c) is not a matroid in the family $(2n,2)$ for $n\geq 2$ (Fig.
$10$).
\end{proposition}

%
%
%
One of the reasons we define the U-independence system for a knot
diagram is so we can examine the interplay between knots and
combinatorial objects like a matroid. Another motivation is to find
new invariants for alternating knots. There is a noteworthy
advantage when a U-independence system for a knot diagram is a
matroid. If it is a matroid, then every maximal U-independent set
has the same cardinality. As a result, every minimal unknotting set
has the same minimal cardinality. In other words, one need only to
find a minimal unknotting set in order to determine the unknotting
number of the knot diagram. This makes it easier to use algorithmic
methods to find the unknotting number of a knot.


\section{Definitions and Examples of Basic Notions}
\subsection{Independence System}
An independence system $(E,I)$  is also called \textit{abstract
simplicial complex} and a \textit{hereditary system}, see \cite{5,12}.
The set $ X $ is called an \textit{independent set} if $X\subset I$ and
called a \textit{dependent set} otherwise. The empty set $\phi $ is
independent and the set $E$ is dependent by definition. Based on the
definition of independence in different contexts, there are a
variety of independence systems. For example, in linear algebra, the
independence system is the usual linear independence, see \cite{13}.
Similarly, for a simple undirected graph the property is
edge-independent, i.e, a set of edges is independent if its induced
graph is acyclic, see \cite{12}. The independent sets of each
independence system $(E,I)$ form different partitions of the ground
set $E$. The partition of $E$ into the smallest number of
independent sets is called a \textit{minimum partition.} The number
of independent sets in a minimum partition of $E$ is called the
$I$\textit{-chromatic number} $\chi (E,I)$ of $(E,I)$ (see \cite{14}
for details).
\subsection{Matroid}
A matroid is a generalization of the linear independence in linear
algebra. The formal definition of a matroid is given here which will be used later in our discussion.
\begin{definition} [\cite{12}]\label{matroid}
 The independence system $(E,I)$ consisting of a family $I$
with subsets of a finite set $E$ is a \textit{matroid }if it
satisfies the exchange property, as previously defined.
To reiterate, the exchange property states that for any two maximal independent sets $M_1$ and
$M_2$ and for every $x\in M_{1}$, there exists a $y\in M_{2}$ such
that $(M_{1}\setminus \{x\})\cup \{y\}$ is also a maximal
independent set.
\end{definition}

The independence systems described in Subsection $2.1$ form
matroids. The first independence system of linearly independent sets
in a vector space is known as the \textit{matric matroid}, and the second
whose independent sets are acyclic sets of edges for a simple
undirected graph is known as the \textit{graphic matroid}, see \cite{9}.
A set of vertices in a simple graph is called a vertex independent set
if no two vertices in the set are adjacent to each other. In general, the vertex
independence system of a simple graph is not a matroid.

\subsection{The Exchange Property for Minimal Unknotting Sets}

If the exchange property holds for all maximal independent sets
(bases) of an independence system, then those bases have the same
cardinality, see \cite{12}. Since minimal unknotting sets are
maximal U-independent sets for a knot diagram, the following remark
is worth mentioning.

\begin{remark}\label{rem} If the exchange property for minimal unknotting sets of a knot diagram holds, then
 all minimal unknotting sets have the same size and
the unknotting number of the diagram can be determined by just
finding a minimal unknotting set.
\end{remark}
\begin{figure}
\begin{center}
\includegraphics [width=6cm]{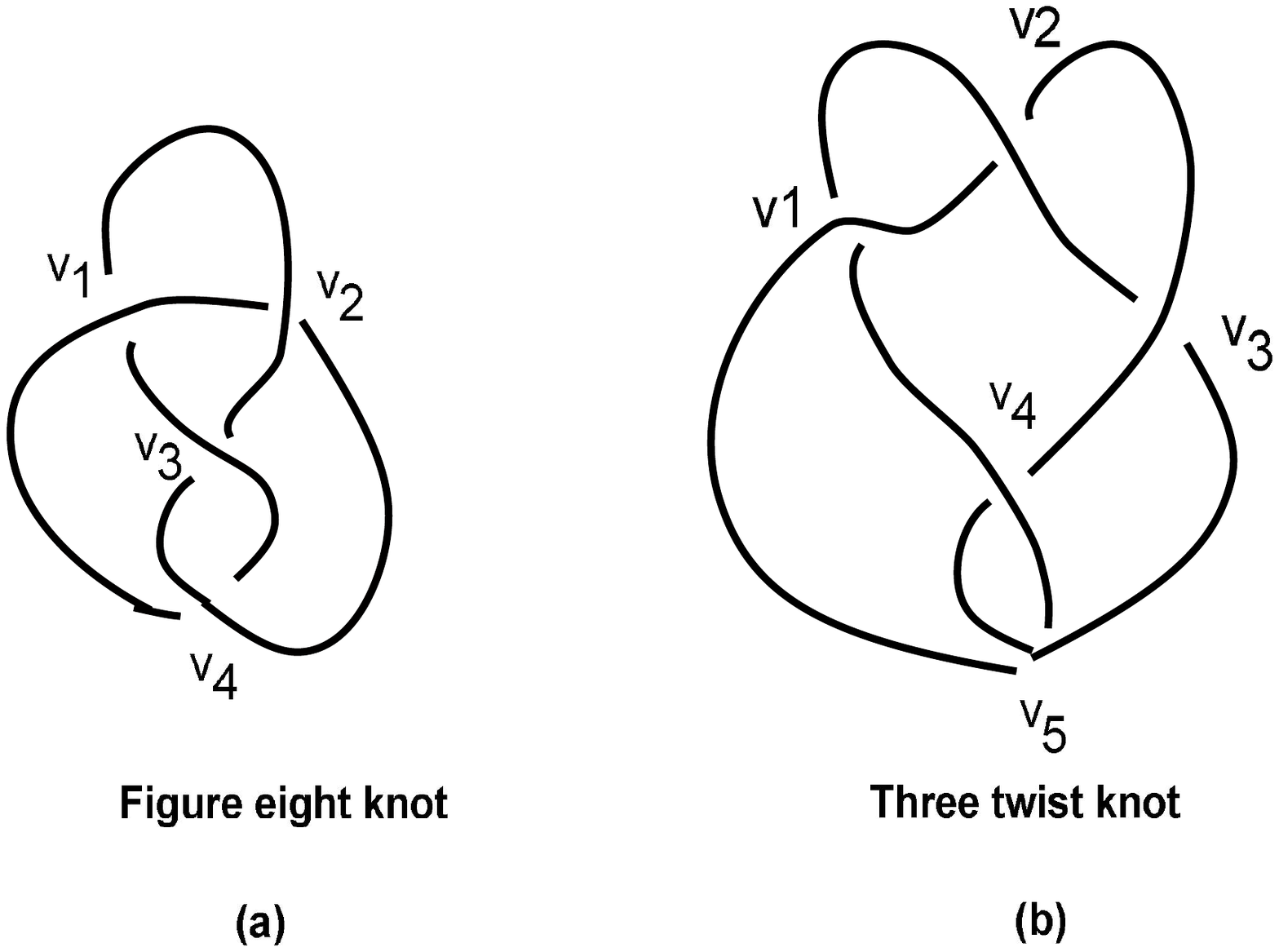}
\caption{} \label{figure1}
\end{center}
\end{figure}
The exchange property for the diagram of the figure eight knot (Fig.
$1(a)$) holds trivially because all the minimal unknotting sets are
of cardinality one. To show that the exchange property does not
hold, we show that there exist two minimal
unknotting sets of different cardinalities. For example, the three
twist knot (Fig. $1(b)$) has the following minimal unknotting sets:
$$ \{v_{4}\},\ \{v_{5}\},\{v_{1},v_{2}\},\{v_{1},v_{3}\} \textrm{
and } \{v_{2},v_{3}\}.$$

\noindent These minimal unknotting sets do not have the same
cardinality, so the exchange property does not hold.\ However, the
exchange property may still not hold for the minimal unknotting sets
that have the same cardinality. For example, the minimal diagram
$8_{3}$ knot (Fig. $2(a)$) has two minimal unknotting sets
$\{v_{1},v_{2}\}$ and $\{v_{5},v_{6}\}$ of the same cardinality.
All possible sets obtained by exchanging elements of these sets are $\{v_{1},v_{5}\},%
\{v_{1},v_{6}\},\{v_{2},v_{5}\},$ and $ \{v_{2},v_{6}\}$, which are
not unknotting sets. For example, when the crossings $v_{1}$ and
$v_{5}$ are switched (Fig. $2(b)$), the knot $8_{3}$ is not
transformed to the unknot. The set $\{v_{1},v_{6}\}$ is also not an
unknotting set (Fig. $2(c)$). Similarly, $\{v_{2},v_{5}\}$ and $
\{v_{2},v_{6}\}$ are not unknotting sets.

\begin{figure}
\begin{center}
\includegraphics [width=11cm]{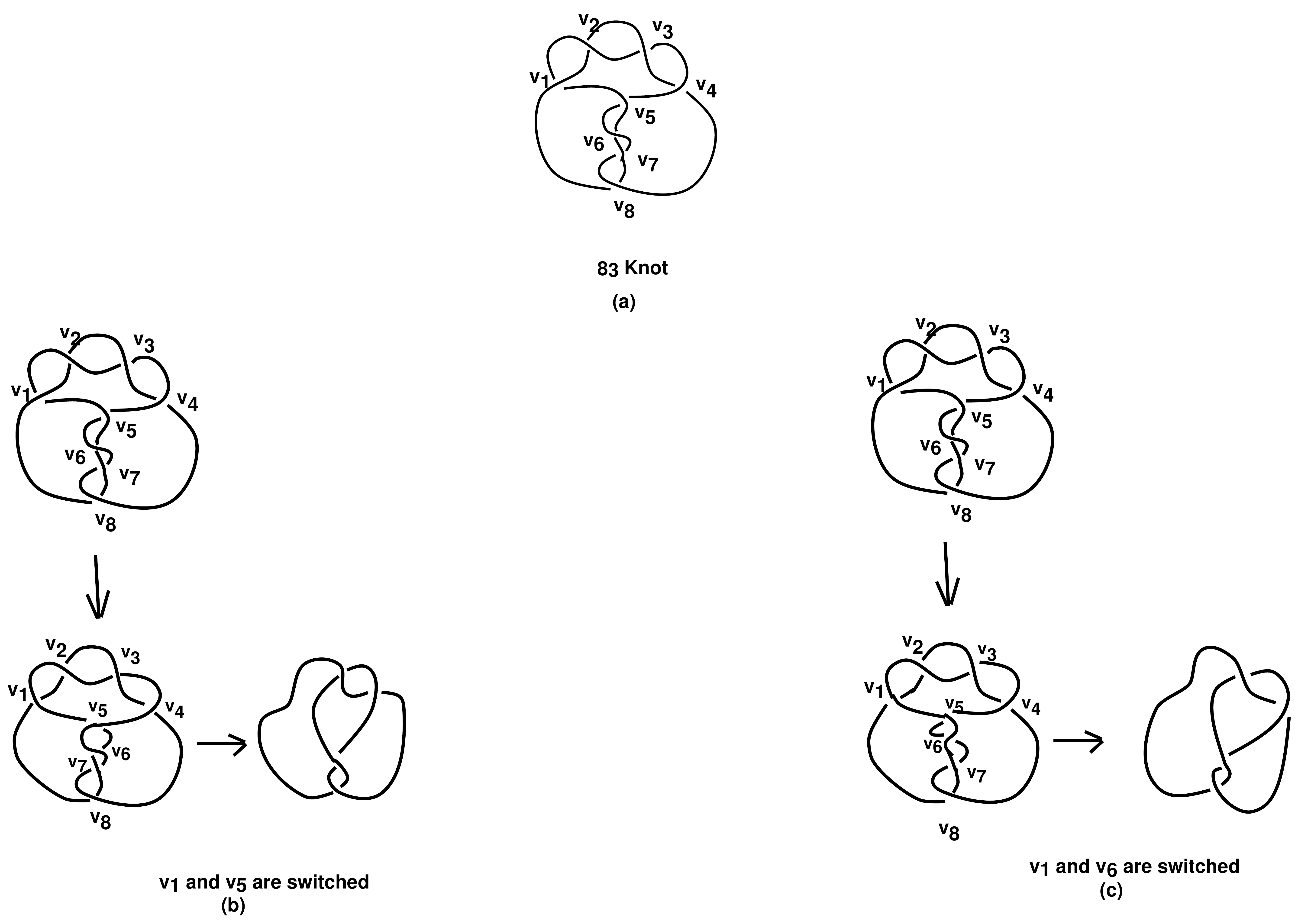}
\caption{} \label{figure2}
\end{center}
\end{figure}
The Table \ref{eqtable} below lists some minimal knot diagrams up to
$8$ crossings that depict their exchange property for minimal
unknotting sets.
\begin{table}[ht]
\caption{}\label{eqtable}
\renewcommand\arraystretch{1.5}
\noindent\[
\begin{array}{|c|c|c|c|}
\hline
Knot & {\small Exchange Prop.\ holds} & Knot & Exchange Prop.\ holds \\
\hline
{3}_1 & yes & {7}_4 & no \\
\hline 4_{1} & yes &
7_{5} & no \\
\hline
 5_{1}& yes & 7_{6} & no \\ \hline
5_{2} & no & 7_{7} & no \\ \hline 6_{1} & no & 8_{1} & no \\
\hline 6_{2} & no & 8_{2} & no \\ \hline 6_{3} & no & 8_{3} & no
\\ \hline 7_{1} & yes & 8_{4} & no \\ \hline 7_{2} & no &
8_{5} & no \\ \hline 7_{3} & no & 8_{6} & no \\ \hline
\end{array}
\]
\end{table}
\FloatBarrier
\section{U-Independence System of a Knot Diagram}

\subsection{Basic Properties}
The idea of converting a minimality in one sense to a maximality in
another sense was first introduced by Boutin (see \cite{3} where
\textit{det-independent} and \textit{res-independent} sets were
defined for determining and resolving sets respectively in simple
graphs.) In this paper, the definition of a U-independent set (see
Definition \ref{key}) is slightly different than the one given in
\cite{3}. This definition is modified to suit our purpose.

The knot diagram of $ 7_{3}$ (Fig. $3(a)$) given in
the Rolfsen knot table \cite{10} has the unknotting number two.
Let $E=\{v_{1},v_{2},v_{3},v_{4},v_{5},v_{6},v_{7}\}$
be the set of all crossings in the minimal diagram (Fig. $3(a)$). Some of the unknotting sets are $W_{1}=\{v_{1},v_{2}%
\},W_{2}=\{v_{1},v_{3}\},W_{3}=\{v_{1},v_{4}\},W_{4}=\{v_{2},v_{3}\},W_{5}=
\{v_{2},v_{4}\}$ and $W_{6}=$ $\{v_{3},v_{4}\}$. All these
unknotting sets are U-independent. For example, $W_{1}\setminus
\{v_{1}\}$ and $W_{1}\setminus \{v_{2}\}$ are not unknotting sets.
There may be other U-independent sets not necessarily unknotting
sets, e.g, $\{v_{1},v_{5}\}$ is not an unknotting
set but U-independent because $\{v_{1},v_{5}\}\setminus \{v_{1}\}$ and $%
\{v_{1},v_{5}\}\setminus \{v_{5}\}$ are not unknotting sets.

A \textit{minimum unknotting set} has the smallest cardinality among
all the minimal unknotting sets of a knot diagram. This smallest
cardinality is actually $u(D)$ of the knot diagram. Every minimum
unknotting set is minimal, but the converse may not always be true. For example, for the knot diagram $7_{3}$ (Fig. $3(a)$), all minimal unknotting sets are:\\
$\{v_{1},v_{2}\},\{v_{1},v_{3}\},\{v_{1},v_{4}\},\{v_{2},v_{3}\},%
\{v_{2},v_{4}\}\ \{v_{3},v_{4}\},\{v_{1},v_{5},v_{6}\},
\{v_{1},v_{5},v_{7}\},$ \\$\{v_{1},v_{6},v_{7}\},\{v_{2},v_{5},v_{6}\},%
\{v_{2},v_{5},v_{7}\},\{v_{2},v_{6},v_{7}\},\{v_{3},v_{5},v_{6}\},%
\{v_{3},v_{5},v_{7}$
$\},\{v_{3},v_{6},v_{7}\},\\ \{v_{4},v_{5},v_{6}\},\{v_{4},v_{5},v_{7}\}$ and $%
\{v_{4},v_{6},v_{7}\}$.

Only $%
W_{1},W_{2},W_{3},W_{4},W_{5}$ and $W_{6}$ are minimum unknotting
sets.
We know that while a knot $K$ has infinite many knot diagrams, it is
not necessarily true that $u(K)$ is always obtained from a minimal
diagram of $K$. In addition, there might be a knot diagram of $K$,
not necessarily a minimal one that has the same unknotting number as
$u(K)$. For the many knots listed in the Rolfsen Table of knots in
\cite{10}, $u(K)$ is the
same for the minimal and other diagrams of $K$. However, for the knot $10_{8}$ ($%
(5,1,4)$ in Conway notation \cite{4}), the minimal diagram (Fig.
$3(b)$) is unknotted by switching at least three crossings with the
minimum unknotting set $\{v_{2},v_{4},v_{6}\}$.
 There is another diagram (Fig. $3(c)$) of $10_{8}$ which
turns to the
unknot by switching only $2$ crossings with a minimum unknotting set $%
\{v_{6},v_{9}^{\prime }\}$.
\begin{figure}
\includegraphics [width=11cm]{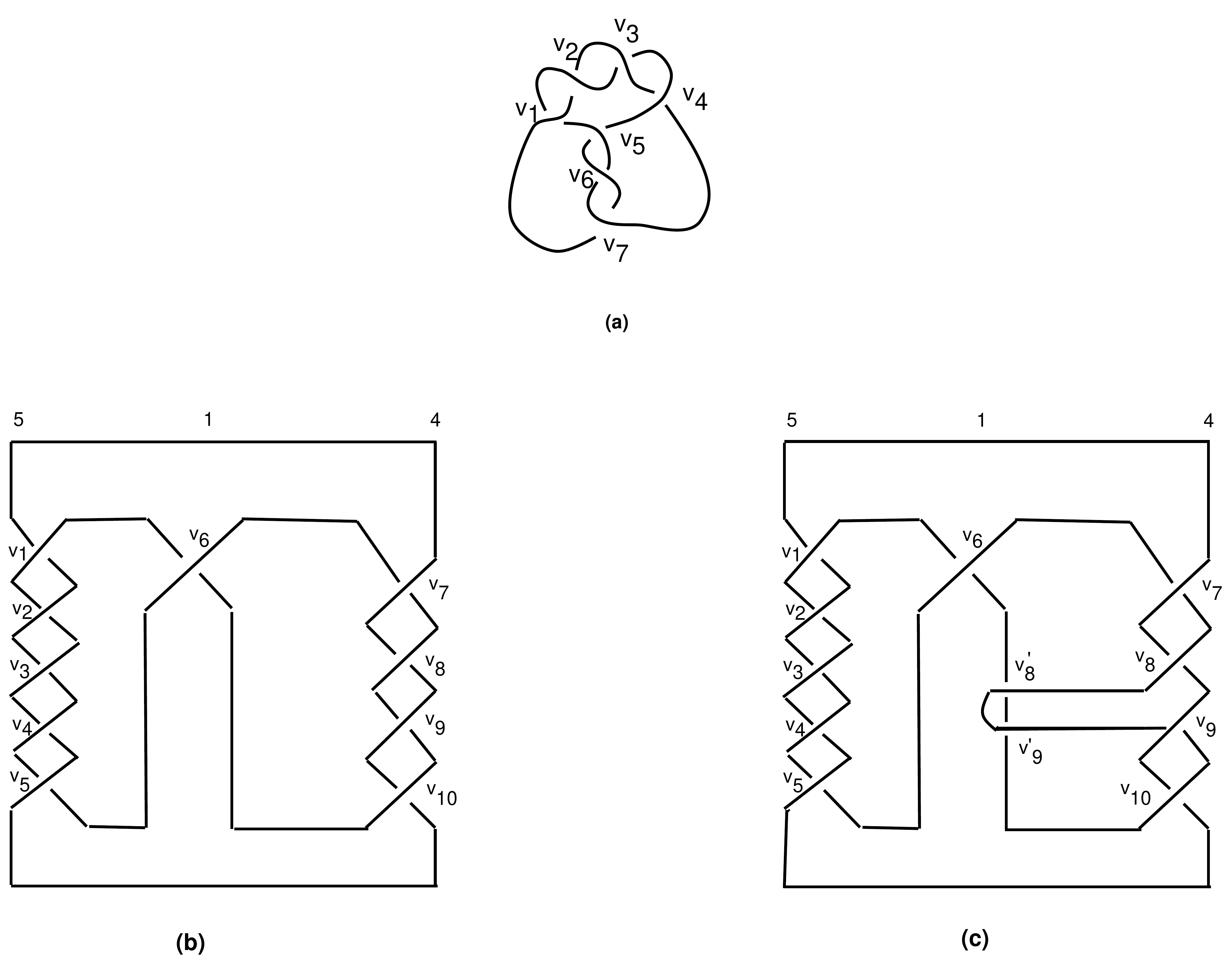}
\caption{} \label{figure3}
\end{figure}
\FloatBarrier The unknotting number of this diagram is actually the
unknotting number of the knot $10_{8}$ (see \cite{2,8}). An
unknotting number, called $u_{min}(K)$, can be defined for each
minimal diagram of a knot $K$, see \cite{11}. Note that for a knot
$K$, the following inequality holds:
$$u(K)\leq u_{min}(K).$$

\subsection{U-independence System as knot invariant}
\begin{definition} [\cite{6a}]
Let ($E_{1},I_{1}$) and ($E_{2},I_{2}$) be two independence systems.
Let there exist a bijection $\varphi :E_{1}\rightarrow E_{2}$ such
that $\varphi (X)\in I_{2}$ if and only if $X\in I_{1}$. Then,
($E_{1},I_{1}$) and ($E_{2},I_{2}$) are said to be isomorphic.
\end{definition}
In order to prove that the U-independence system is a knot invariant
for an alternating knot, the following well-known conjecture of Tait
(proved in \cite{TC} by Menasco and Thistlethwaite) is needed.
\begin{theorem}([The Tait flyping conjecture])\label{tait}
Given reduced alternating diagrams $D_1$, $D_2$ of a knot (or link), it is then possible to transform $D_1$ to $D_2$ by a sequence of
\emph{flypes} (Fig. $4$).
\end{theorem}

\textbf{Proof of Theorem \ref{iso}.}
 Let $v_i$ be a crossing in the diagram $D_1$. Apply the flype (Fig. $4$) to $D_1$ to remove the crossing
 $v_i$ and create a new crossing with the same label $v_i$. More precisely, the tangle (the shaded disc in
Fig. $4$) is turned upside-down to map the crossing (one to its
left) to the crossing (one to its right). During the application of
the flype all the unknotting/not unknotting sets of the diagram
$D_1$ are preserved. Consequently, all the U-independent sets are
preserved in the process. By Theorem \ref{tait}, the diagram $D_1$
can be converted to $D_2$, through a sequence of the flypes,
preserving the U-independent sets. As a result, an isomorphism
$\varphi$ between  $(E_{1},I_{1})$ and ($E_{2},I_{2})$ is
established. \qed

\begin{figure}
\includegraphics [width=10cm]{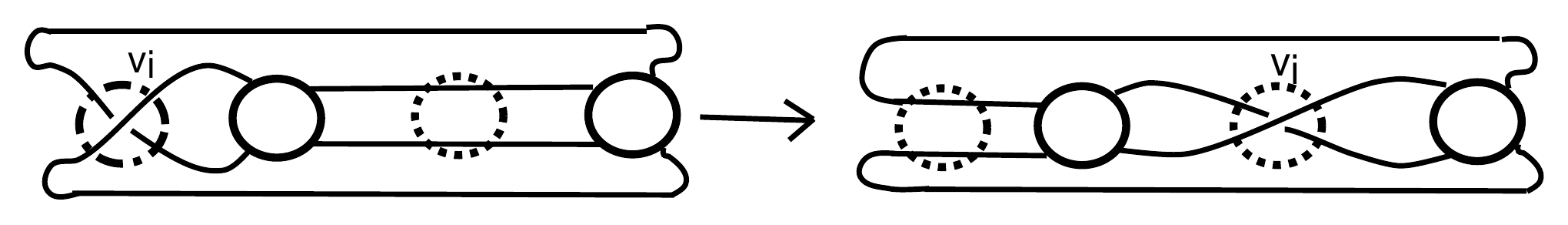}
\caption{flype} \label{figure3}
\end{figure}
 Theorem \ref{iso} further states that the U-independence system (defined for a reduced alternating diagram $D$
of a knot $K$) itself and all its invariants are knot invariants.
The number $u_{min}(K)$ can also be defined as the cardinality of a
U-independent set which is also a minimum unknotting set. The number
is not a complete invariant, i.e., there are non-isotopic knots
having the same $u_{min}(K)$.\ However, other invariants of the
${\Large U}$-independence systems of non-isotopic alternating knots
may distinguish them where $u_{min}(K)$ fails to do so. Here are two
such examples.

\noindent\textbf{The Number of U-independent Sets of a Fixed
Cardinality}
\begin{example}
Consider the reduced diagram (Fig. $5(a)$) of knot $6_{1}$ and the
reduced diagram (Fig. $5(b)$) of knot $6_{2}$. For the knot $6_{1},$ the set of all crossings $E=%
\{v_{1},v_{2},v_{3},v_{4},v_{5},v_{6}\}$ is divided into two disjoint sets $%
A $ and $B$: the set $A=\{v_{1},v_{2},v_{3},v_{4}\}$ and
$B=\{v_{5},v_{6}\}$. In $A$, no single cross switching turns the
knot into the unknot. In contrast, when any crossing in $B$ is
switched, the knot is unknotted. All possible subsets of
$\{v_{1},v_{2},v_{3},v_{4}\}$ of cardinality two are minimal
unknotting sets. Furthermore, every subset of cardinality three,
four, or five contains an unknotting set. Thus, all the
U-independent
sets are $\{v_{1}\},\{v_{2}\},\{v_{3}\},\{v_{4}\},\{v_{5}\},\{v_{6}\},%
\{v_{1},v_{2}\},\\\{v_{1},v_{3}\}, \{v_{1},v_{4}\},$
$\{v_{2},v_{3}\}, \{v_{2},v_{4}\},$ $\{v_{3},v_{4}\}$. There are six
U-independent sets of size $2.$

\begin{figure}
\includegraphics [width=10cm]{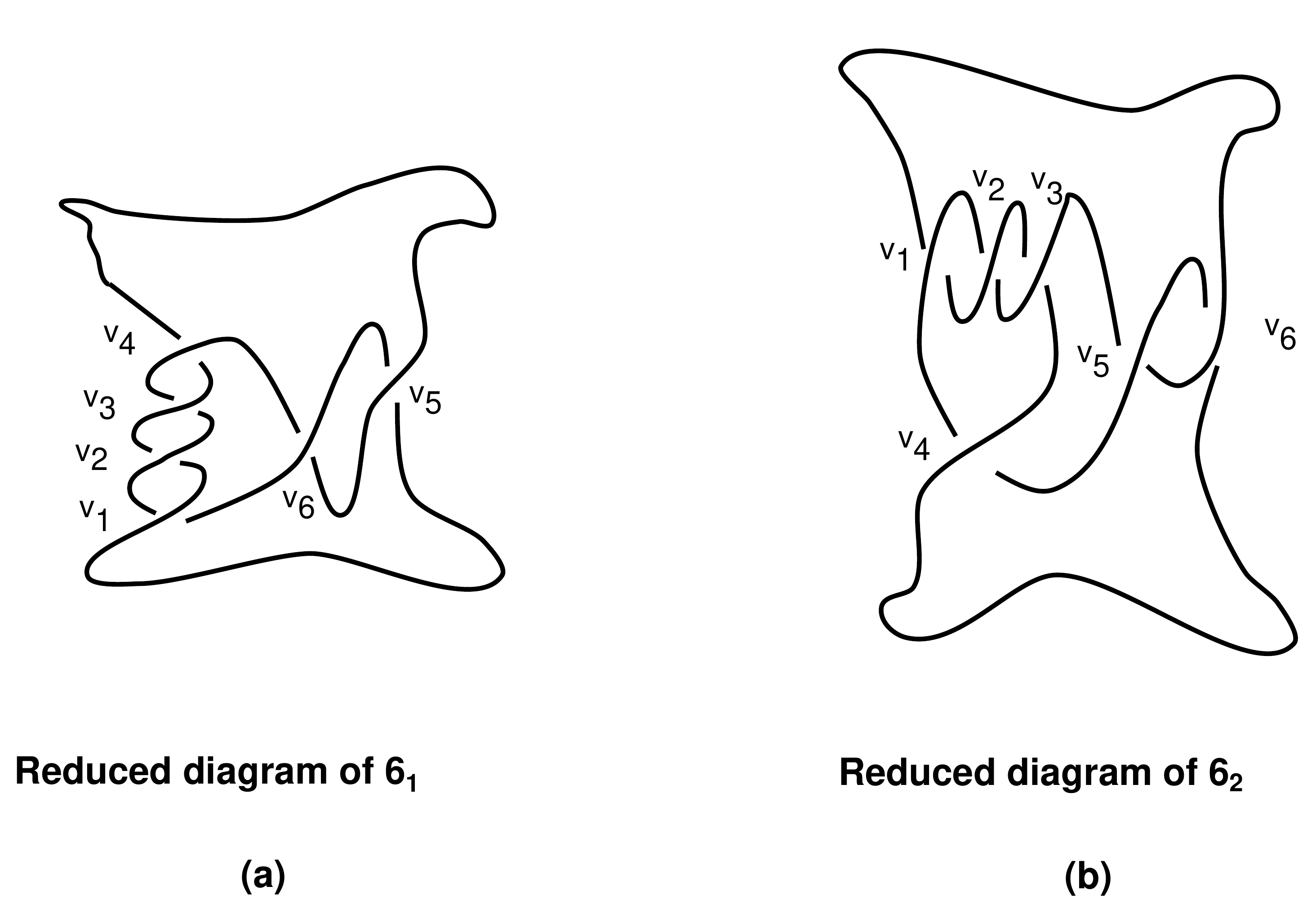}
\caption{Reduced diagram of $6_{1}$ and $6_{2}$} \label{figure3}
\end{figure}
For the knot $6_{2}$, the set
$E=\{v_{1},v_{2},v_{3},v_{4},v_{5},v_{6}\}$ (Fig. $5(b)$) is
divided into three disjoint sets $A$, $B$, and $C$: in the set $%
A=\{v_{1},v_{2},v_{3}\}$, there is no crossing in $A$ which turns
the knot to the unknot; the set $B=\{v_{4}\}$ is an unknotting set; and the set $%
C=\{v_{5},v_{6}\}$ contains no unknotting set. When any two crossings from $%
A\cup B$ are switched, the knot is unknotted. However, there is no
unknotting set of cardinality 2 in $B\cup C$. Furthermore, every
subset of cardinality $3$, $4$, or $5$ contains an unknotting set.
Thus, the U-independent sets are $\{v_{1}\},$
$\{v_{2}\},\{v_{3}\},\{v_{4}\},\{v_{5}\},%
\{v_{6}\},\{v_{1},v_{2}\},\{v_{1},v_{3}\},\{v_{1},v_{5}\},\{v_{1},v_{6}\},%
\{v_{2},v_{3}\},\{v_{2},v_{5}\},\\\{v_{2},v_{6}\},\{v_{3},v_{5}\},\{v_{3},v_{6}\},\{v_{5},v_{6}\}$.\
There are $10$ U-independent sets of size $2$. The knots $6_1$ and
$6_2$ are distinguished by the number of U-independent sets of
cardinality $2$.
\end{example}
\noindent\textbf{The $I$\textit{-chromatic Number}} \\ The set of
crossings of a reduced knot diagram $D$ can be partitioned into
U-independent sets and the minimum number of such U-independent sets
gives a minimum partition of $E$. The number of U-independent sets
in a minimum partition of $E$ gives the $I$\textit{-chromatic number
}$\chi (E,I)$. The number $\chi (E,I)$ for $D$ can be used as a knot
invariant in combination with $u_{min}(K)$. In other words, two
alternating knots can be distinguished by $\chi (E,I)$ if the knots
have the same $u_{min}(K)$.
\begin{example}
Consider the reduced diagram (Fig. $6(a)$) of knot $7_{2}$ and the
reduced diagram (Fig. $6(b)$) of knot $7_{7}$. For the knots
$7_{2},$ $E= \{v_{1},v_{2},v_{3},v_{4},v_{5},v_{6},v_{7}\}$ (Fig.
$6(a)$) is divided into two
disjoint subsets $A$ and $B$: $A=\{v_{1},v_{2},v_{3},v_{4},v_{5}\}$ and $B$ $%
=\{v_{6},v_{7}\}$.\ The set $A$ contains no unknotting set of
cardinality one and two. Each subset of $A$ with cardinality three is
a minimal unknotting set.\ Every crossing in $B$ unknots the knot,
but $B$ itself is not an unknotting set. Every subset of $E$
containing $\{v_{6}\}$ or $\{v_{7}\}$ is not a minimal unknotting
set.\ Furthermore, every set of cardinality four, five and six
contains an unknotting set.\ Thus, the U-independent sets are:\\
$\{v_{1}\},\{v_{2}\},\{v_{3}\},\{v_{4}\},\{v_{5}\},\{v_{6}\},\{v_{7}\},%
\{v_{1},v_{2}\},\{v_{1},v_{3}\},\{v_{1},v_{4}\},\{v_{1},v_{5}\},\{v_{2},$%
$v_{3}\},\\ \{v_{2},v_{4}\},\{v_{2},v_{5}\},
\{v_{3},v_{4}\},\{v_{3},v_{5}\},%
\{v_{4},v_{5}\},\{v_{1},v_{2},v_{3}\},\{v_{1},v_{2},v_{4}\},%
\{v_{1},v_{2},v_{5} \},$ \\ $\{v_{1},v_{3},v_{4}\},
\{v_{1},v_{3},v_{5}\},\{v_{1},v_{4},v_{5}\},%
\{v_{2},v_{3},v_{4}\},\{v_{2},v_{3},v_{5}\},\{v_{2},v_{4},v_{5}\},%
\{v_{3},v_{4},$ $v_{5}\}.$ A minimum partition is $\{\{v_{1},v_{2},v_{3}\},\{v_{4},v_{5}\},\{v_{6}\},%
\{v_{7}\}\}$ and $\chi (E,I)=4.$
\begin{figure}
\includegraphics [width=8cm]{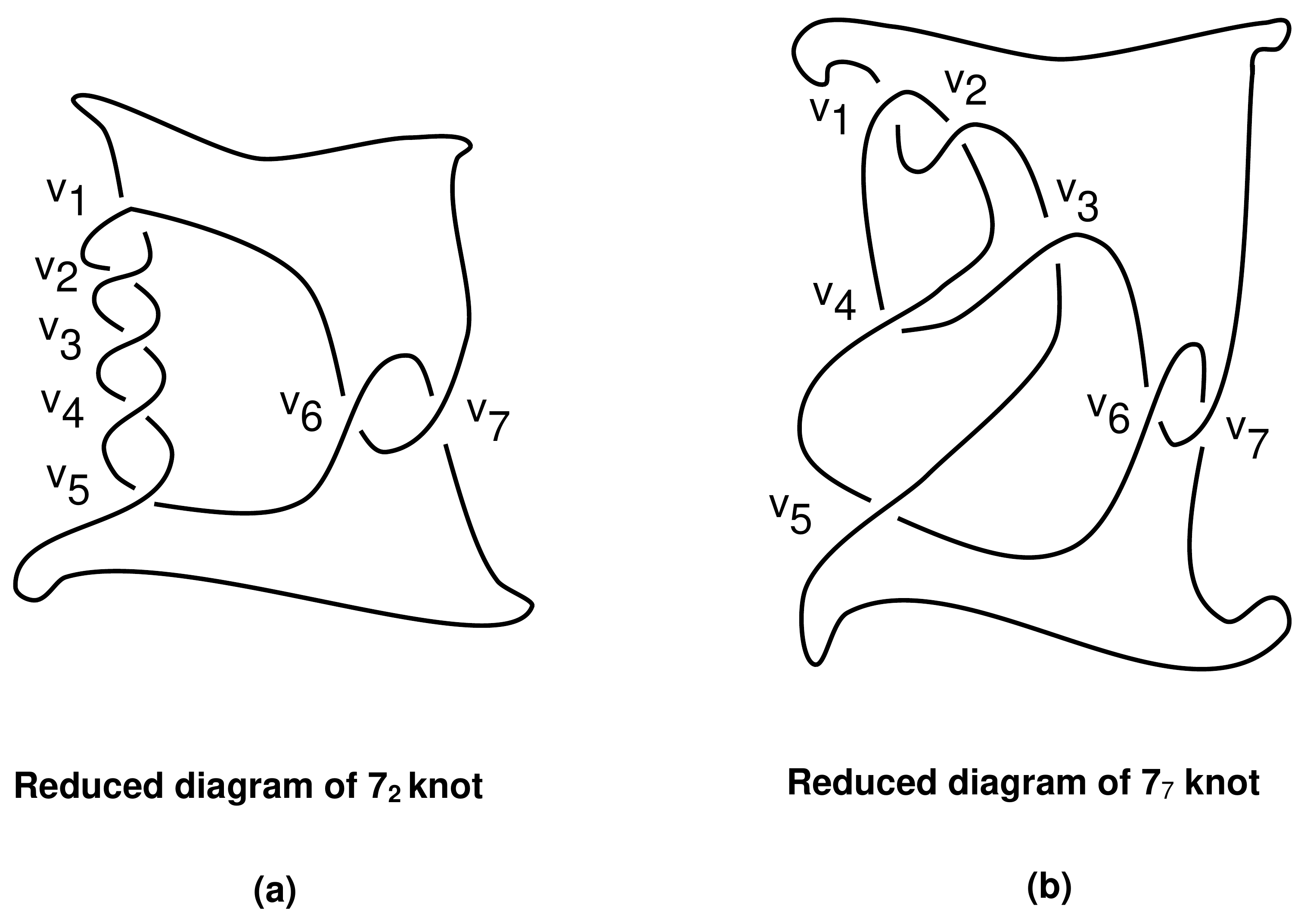}
\caption{Reduced diagram of $7_2$ and $7_{7}$ knot} \label{figure3}
\end{figure}
\FloatBarrier
For the knot $7_{7},$ the set $E=\{v_{1},v_{2},v_{3},v_{4},v_{5},v_{6},v_{7}%
\}$ (Fig. $6(b)$) is divided into three disjoint subsets $A,$ $B,$ and $C$. The set $A$ $%
=\{v_{1},v_{2},v_{3}\}$, $B$ $=\{v_{4},v_{5}\},$ and $C$
$=\{v_{6},v_{7}\}$. In the set $A$, no unknotting set of cardinality
one exists; for cardinality two, all of the sets are unknotting sets except for $\{v_{1},v_{2}\}$.\ In $B$, $%
\{v_{4}\}$ and $\{v_{5}\}$ are unknotting sets but $B$ itself is not
an unknotting set.\ In $C$, neither a set of cardinality one nor $C$
itself is an unknotting set. Every set of cardinality three, four,
five and contains an unknotting set. Thus, the U-independent sets
are: $\{v_{1}\},\{v_{2}
\},\{v_{3}\},\{v_{4}\},\{v_{5}\},\{v_{6}\},\{v_{7}\},$
$\{v_{1},v_{2}\},$ $\{v_{1},v_{3}\},\{v_{1},v_{6}\},\{v_{1},v_{7}\},
\{v_{2},v_{3}\}, \\ \{v_{2},v_{6}\},\{v_{2},v_{7}\},\{v_{3},v_{6}\},
\{v_{3},v_{7}\},\{v_{6},v_{7}\}.$ \\
A minimum partition is
$\{\{v_{1},v_{2}\},\{v_{3}\},\{v_{4}\},\{v_{5}\}, \{v_{6},v_{7}\}\}$
and $\chi (E,I)=5.$ Hence, the knots $7_{2}$ and $7_{7}$ are
distinguished by $\chi (E,I)$.
\end{example}

\section{U-independence as a Matroid}
\subsection{Family $(2n+1, 1, 2n)$}
For a knot $K$ in the family $(2n+1, 1, 2n)$ with $n\geq 2$,
$u(K)=n<u_{min}(K)=n+1$, see \cite{2}. The unknotting number of a
knot in this family can be obtained from the diagram (Fig. $7$).

\textbf{Proof of Proposition \ref{fig15} Part $a)$.} For the knot
diagram $D$ (Fig. $7$) with $u(D)=n$, there are two minimal
unknotting sets $\{w,u_{3}^{\prime },u_{5}^{\prime },u_{7}^{\prime
},\ldots u_{2n-1}^{\prime }\}$ and $\{v_{2},v_{4},v_{6},\ldots,
v_{2n},w\}$ of cardinalities $n$ and $n+1$ respectively.
Consequently, there are two maximal U-independent sets of different
cardinalities. By Remark \ref{rem}, the U-independence system is not
a matroid. \qed
\begin{figure}
\includegraphics [width=5cm]{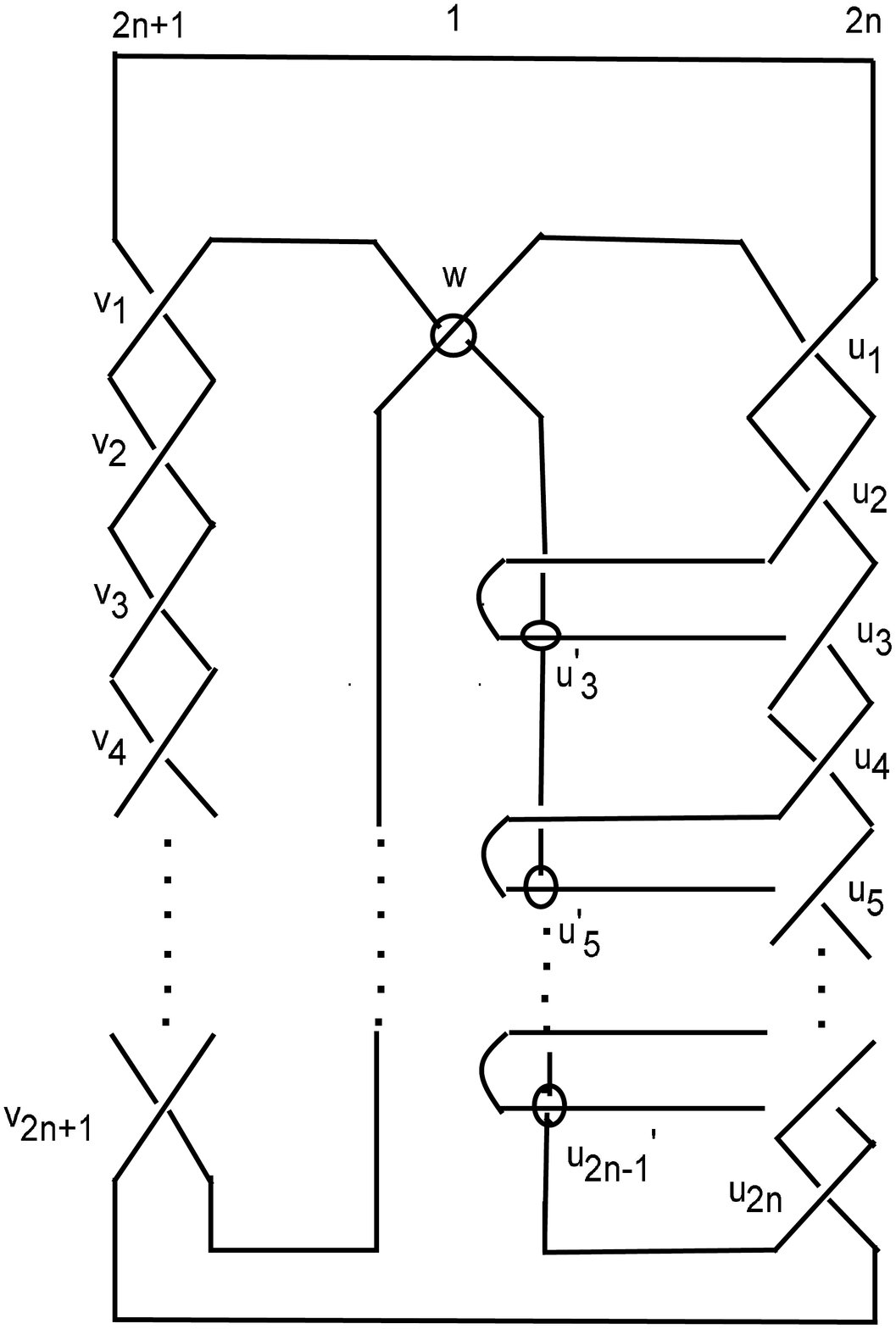}
\caption{$(2n+1,1,2n)$} \label{figure3}
\end{figure}

\subsection{Family $(2n+1)$}

The following result may be known to an expert in knot theory.
Anyhow, it is proved here for the sake of completion.
\begin{lemma}\label{unn}
A knot diagram $D$ in the family $(2n+1)$  has $u(D)=n$.
\end{lemma}

\begin{proof}
Apply induction on $n$. For $n=1$, $(2n+1)$ is a reduced diagram of
trefoil knot with $u(D)=1$. Suppose $u(D)=m$ for $(2m+1)$.
For $n=m+1,(2(m+1)+1)=(2m+3)$ is a family of knot diagram with
$2m+3$ alternating crossings (Fig. $8$). When the crossing
$v_{2m+3}$ is switched, the crossing $v_{2m+2}$ is also killed and
the knot diagram $(2m+1)$ is obtained (Fig. $8$). By induction,
$u(D)\leq m+1$. The knot diagram $(2m+1)$ can not be unknotted by
fewer than $m$ crossings because if $m-1$ crossings are switched,
then $2(m-1)$ alternating crossings are untangled and a reduced
diagram of trefoil knot is obtained. It follows that the unknotting
number of $(2m+3)$ is $m+1$.
\end{proof}
\begin{figure}
\includegraphics [width=8cm]{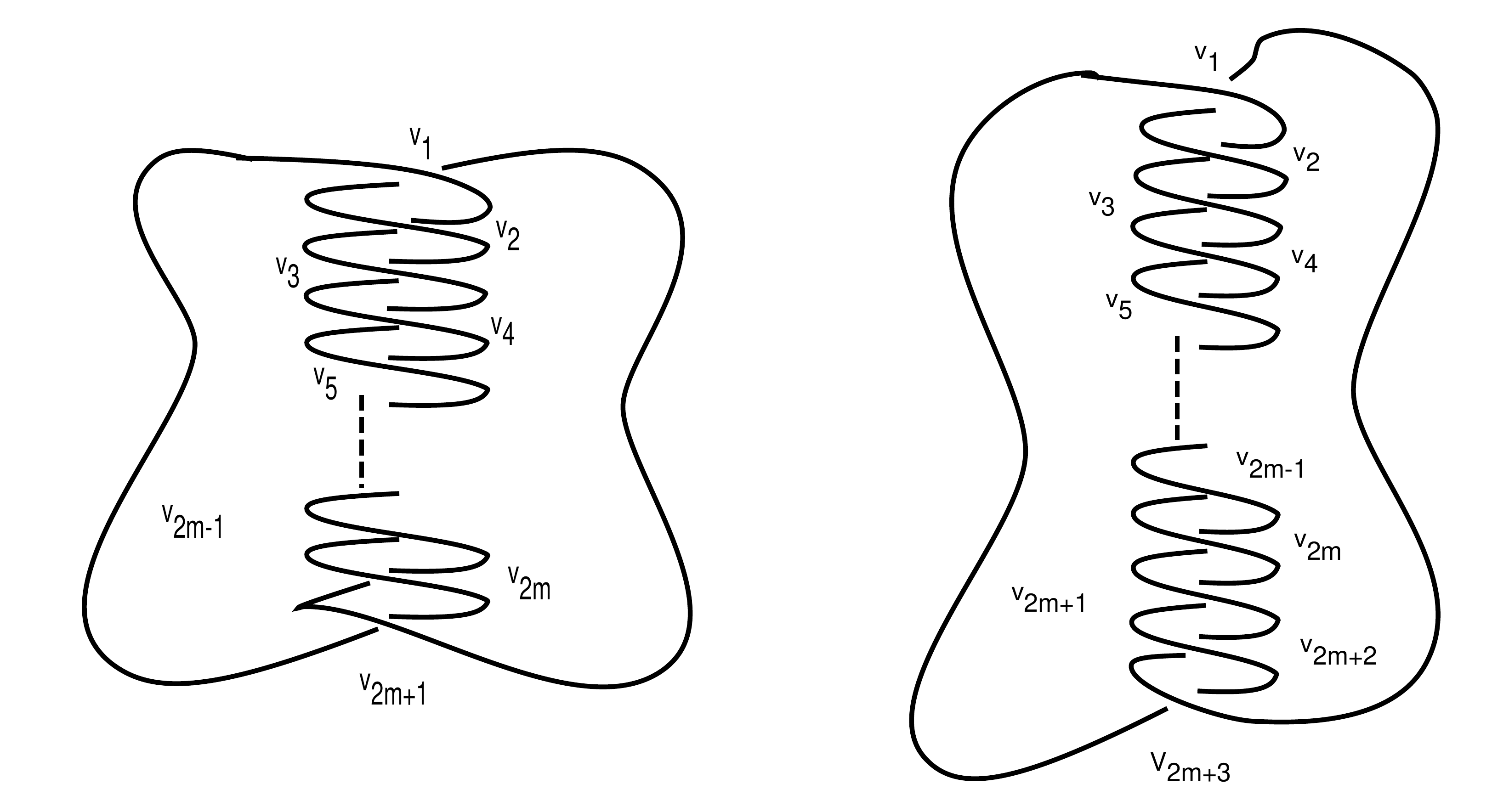}
\caption{$(2m+1)$ and $(2m+3)$} \label{figure3}
\end{figure}

\textbf{Proof of Proposition \ref{fig15} Part (b)} The diagram (Fig.
$8$) has the property that every subset $A\subset E=
\{v_{1},v_{2},v_{3},\ldots,v_{2n},v_{2n+1}\}$ with $\mid A\mid=n$ is
an unknotting set. By Lemma \ref{unn}, the set $A$ must be a minimal
unknotting set (a maximal U-independent set). There is no maximal
U-independent set of cardinality $<n$. Also, there is no
U-independent set $B$ of cardinality $>n$ because $B$ contains an
unknotting set of cardinality $n$. It follows that a U-independent
set is maximal if and only if it is a minimal unknotting set and of
cardinality $n$. The exchange property holds for all maximal
U-independent sets as every subset of $E$ of cardinality $n$ is a
maximal U-independent. Hence, the U-independence system is a matroid
by Definition \ref{matroid}. \qed
\subsection{Family $(2n,2)$}

\begin{figure}
\begin{center}
\includegraphics [width=5cm]{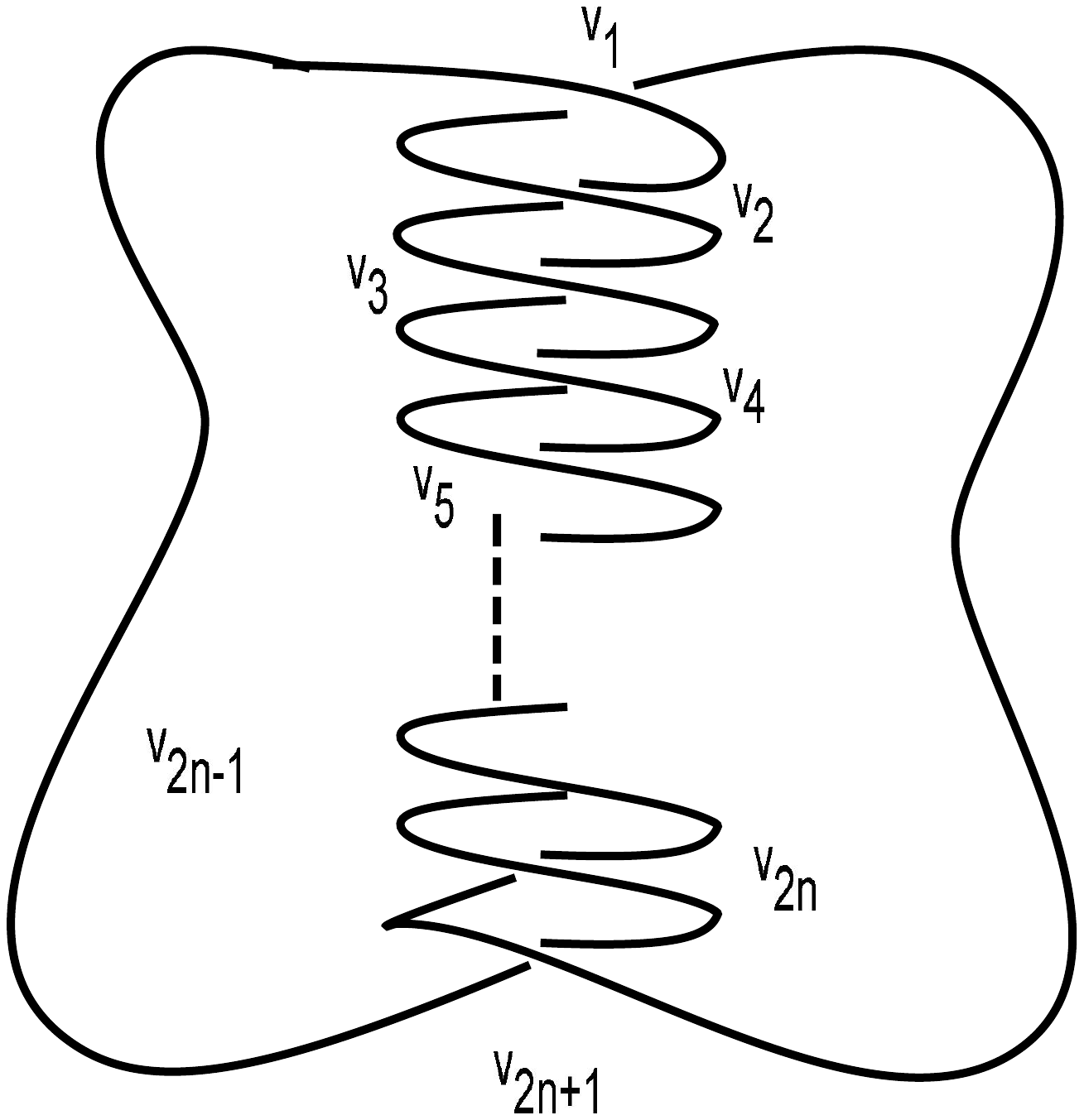}
\caption{$(2n+1)$} \label{figure3}
\end{center}
\end{figure}

For $n\geq 1,$ each diagram $D$ (Fig. $10$) has $u(D)=1$ and its
U-independence system is not a matroid except for the figure eight
knot, i.e., when $n=1$.

\textbf{Proof of Proposition \ref{fig15} Part $c)$.} For the diagram
(Fig. $10$), the sets $\{w\}$ and $ \{v_{1},v_{2},v_{3},v_{4},\ldots
,v_{n}\}$ are minimal unknotting sets of cardinality $1$ and $n$
respectively. Thus, there are two maximal U-independent sets having
different cardinalities. By Remark \ref{rem}, the U-independence
system is not a matroid. \qed
\begin{figure}
\begin{center}
\includegraphics [width=4cm]{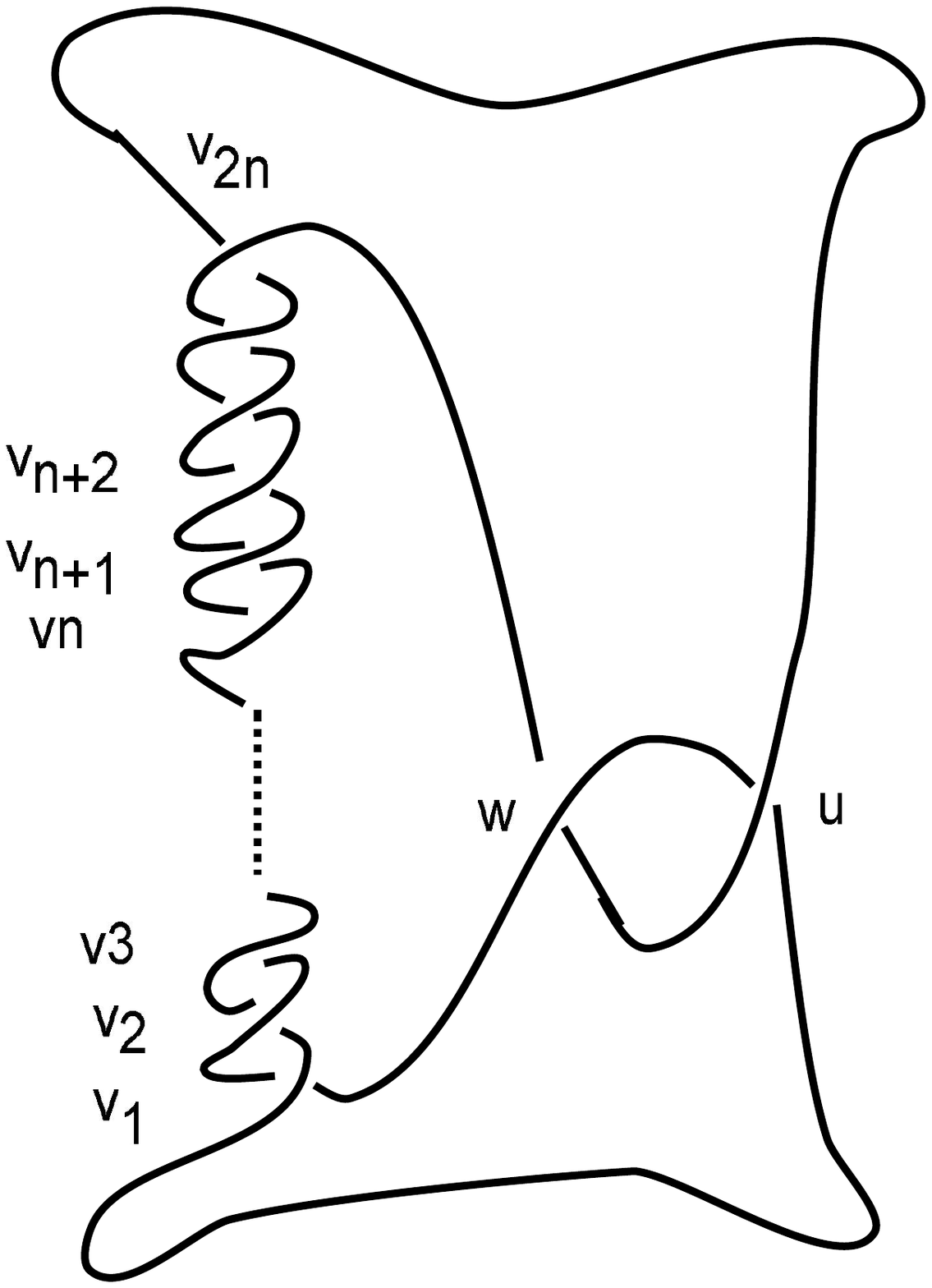}
\caption{$(2n,2)$} \label{figure3}
\end{center}
\end{figure}
\FloatBarrier
\section{Conclusion}
This completes our introduction to U-independence system of a
classical knot. On the same lines, independence systems can also be
defined for a knot diagram with respect to other invariants like
bridge numbers and algebraic unknotting numbers.\ The corresponding
invariants of these independence systems may also be used as knot
invariants in combination with these invariants.\ Every independence
system $(E,I)$ is an abstract simplicial complex, see \cite{5}.\
Therefore, the homology of $(E,I)$ can be investigated for finer
invariants of the corresponding knots.\ Similarly independence
systems can be associated and studied for virtual knots, see
\cite{6}.

There is much that implores for further investigation. For example,
one can show that the U-independence systems for reduced alternating
diagrams of $6_1$ and $6_3$ are isomorphic. The knots $6_1$ and
$6_3$ are not mirror images of each other. We can then ask the
following
open question:\\
 \textbf{Question.} Does there exist two non-isotopic
alternating knots (not the mirror image of each other) that have the
same number of crossings in their reduced alternating diagram and
the same $u_{min}(K)\geq 2$ with isomorphic U-independent systems?
\section{Acknowledgement}
The authors would like to thank Professor Colin Adams and his
student Jonathan Deng for their help in improving the exposition and
language of this paper. Specially Professor Colin, who was very kind
to clarify some notions in knot theory.
\bibliographystyle{amsplain}

\begin{thebibliography}{10}
\bibitem{1} C.\ C.\ Adams, \textit{The Knot Book: An Elementary Introduction to the
Mathematical Theory of Knots}, W.H.\ Freeman and Company $1994$.

\bibitem{2} J.\ A.\ Bernhard, Unknotting numbers and minimal knot diagrams,
\textit{J. knot Theory Ramifications} \textbf{3}$(1992)$ $1$-$5$.

\bibitem{3} D.\ L.\ Boutin, Determining sets, resolving sets and the
exchange property, \textit{Graphs Combin.}\ \textbf{25}$(2009)$
$789$-$806$.

\bibitem{4} J. H. Conway, On enumeration of knots and links,
\textit{Proc. Conf. Oxford} $(1967)$ $329$-$258$.

\bibitem{5} B. Korte, L. Lov'{a}sz, R. Schrader, \textit{Greedoids},
Springer-Verlag, Berlin, $1991$.

\bibitem{K6} L. H. Kauffman, State models and the Jones polynomial. \textit{Topology}
$\mathbf{26}(1987)$ $395$–$407$.

\bibitem{6} L.\ H.\ Kauffman, Virtual knot theory, \textit{European J. Combin.} $\mathbf{20}(1999)$ $663$-$690$.

\bibitem{TC} W.\ W.\ Menasco, M. B. Thistlethwaite, The Tait flyping conjecture, \textit{Bull. Amer. Math. Soc. (N.S.)}\ $
\mathbf{20}2(1991)$ $ 403$-$412$.

\bibitem{6a} O. Melnikov, V. Sarvanov, R.I. Tyshkevich, V. Yemelichev, I.E. Zverovich,
\textit{Excercises in Graph Theory}, Springer $1998$.

\bibitem{8} Y.\ Nakanishi, Unknotting numbers and knot diagrams with the the
minimum crossings, \textit{Mathematics Seminar Notes, kobe
University}, \textbf{11}$(1983)$ $257$-$258$.

\bibitem{9} J. G. Oxley, \textit{Matroid Theory}, Oxford University Press, Oxford, $%
1992$.

\bibitem{10} D. Rolfsen, \textit{Knots and links, Berkeley}, Calif, Publish or
Perish, $1976$.

\bibitem{11} A.\ Stoimenov, On the unknotting number of minimal diagrams,
\textit{Math. Comp.} \textbf{72}$(2003)$ $2043-2057$.

\bibitem{12} D. B. West, \textit{Introduction to Graph Theory}, Prentice Hall, $%
2001 $.

\bibitem{13} H.\ Whitney, On the abstract properties of linear dependence, \textit{Amer. J. Math.} \textbf{57}$(1935)$ $509-533$.

\bibitem{14} S.\ Zhou, Minimum partition of an independence system into
independent sets, \textit{Discrete Optim.} \textbf{6}$(2009)$
$125-133$.


\end{thebibliography}

\end{document}